\documentclass[11pt,leqno]{amsart}

\usepackage{latexsym,enumerate,graphicx}
\usepackage{amsmath,amsthm,amsopn,amstext,amscd,amsfonts,amssymb}
\usepackage[T1]{fontenc}
\usepackage{fullpage,yhmath}
\usepackage{hyperref}
\usepackage{eucal}
\usepackage{color}
\usepackage{verbatim}
\usepackage[normalem]{ulem}
\usepackage{frcursive}
\usepackage{pgfplots}
\usepackage{tikz}
\usepackage{cleveref}

\numberwithin{equation}{section}

\newtheorem{theorem}{Theorem}
\newtheorem{lemma}{Lemma}

\newtheorem{proposition}{Proposition}
\newtheorem{remark}{Remark}

\newtheorem{definition}{Definition}

\numberwithin{theorem}{section}
\numberwithin{corollary}{section}
\numberwithin{lemma}{section}
\numberwithin{definition}{section}
\numberwithin{proposition}{section}
\numberwithin{remark}{section}

\newcommand{\RR}{\mathbb R^N}
\newcommand{\R}{\mathbb R}

\newcommand{\medint}{-\kern  -,375cm\int}
\newcommand{\dint}{\displaystyle\int}

\newcommand{\supp}{\mathop{\mathrm{supp}}\nolimits}

\newcommand{\inte}{\int\!\!\!\!\int}

\newcommand{\di}{\mathop{\mathrm{d}\!}}

\makeatletter
\@namedef{subjclassname@2020}{  \textup{2020} Mathematics Subject Classification}
\makeatother

\begin{document}
\title[Fractional Isoperimetric Inequalities]{First Eigenvalue and Torsional Rigidity:\\
Isoperimetric Inequalities for the Fractional Laplacian}

\author{Barbara Brandolini$^1$}
\footnotetext[1]
{\textsc{Dipartimento di Matematica e Informatica, Universit\`a degli Studi di Palermo, via Archirafi 34, 90123 Palermo, Italy.}

\noindent\textit{Email address:} {\tt barbara.brandolini@unipa.it}}

\author{Ida de Bonis$^2$}
\footnotetext[2]
{\textsc{Dipartimento di Pianificazione, Design, Tecnologia dell'Architettura, Sapienza Universit\`a di Roma, via Flaminia 72, 00196 Roma, Italy.}

\noindent\textit{Email address:} {\tt ida.debonis@uniroma1.it}}

\author{Vincenzo Ferone$^3$}
\footnotetext[3]
{\textsc{Dipartimento di Matematica e Applicazioni ``Renato Caccioppoli'', Universit\`a degli Studi di Napoli Federico II, Via Cintia, Complesso Universitario Monte S. Angelo, 80143 Napoli, Italy.}

\noindent\textit{Email address:} {\tt ferone@unina.it} \textrm{(corresponding author)}}

\author{Gianpaolo Piscitelli$^4$}
\footnotetext[4]
{\textsc{Dipartimento di Scienze Economiche Giuridiche Informatiche e Motorie, Universit\`a degli Studi di Napoli Parthenope, Via Guglielmo Pepe, Rione Gescal, 80035 Nola (NA), Italy.}

\noindent\textit{Email address:} {\tt gianpaolo.piscitelli@uniparthenope.it}}

\author{Bruno Volzone$^5$}
\footnotetext[5]
{\textsc{Dipartimento di Matematica, Politecnico di Milano, Piazza Leonardo da Vinci 32, 20133 Milano, Italy.}

\noindent\textit{Email address:} {\tt bruno.volzone@polimi.it}}

\setcounter{tocdepth}{1}


\begin{abstract}
We present a fractional counterpart of a generalized Kohler-Jobin inequality, showing that, among all bounded, open sets $\Omega\subset \R^N$ with Lipschitz boundary, having the same fractional torsional rigidity, the first Dirichlet eigenvalue $\lambda_1(\Omega)$ of the fractional Laplacian attains its minimum on balls. With the same arguments we also establish a reverse H\"older inequality for an eigenfunction corresponding to $\lambda_1(\Omega)$.


\end{abstract}

\maketitle

%%%%%%%%%%%%%%%%%%%%%%%%

\section{Introduction}

It is well-known that, for a given, bounded open set $\Omega \subset{\mathbb{R}^N}$, one can define the following quantities
\begin{equation}
    \label{localei}
\lambda_1(\Omega) = \min_{\xi \in H_0^1(\Omega) \setminus \{0\}} \dfrac{\displaystyle\int_{\Omega} |D\xi|^2 \di x}{\displaystyle\int_{\Omega} \xi^2 \di x}
\end{equation}
and
\begin{equation}
    \label{localto}
T(\Omega)= \max_{\eta \in H_0^1(\Omega) \setminus \{0\}} \dfrac{\displaystyle\left(\int_{\Omega} |\eta| \di x\right)^2}{\displaystyle\int_{\Omega} |D\eta|^2 \di x}
\end{equation}
which are known as the principal frequency (the first one) and the torsional rigidity (the second one) of $\Omega$. 
Both quantities are realized by the solutions to some Dirichlet boundary problems. Namely, the minimum in \eqref{localei} is achieved by an eigenfunction for the problem
\begin{equation*}
\begin{cases}
-\Delta u = \lambda u & \text{in } \Omega, \\
u = 0 & \text{on } \partial\Omega,
\end{cases}
\end{equation*}
corresponding  to the first (smallest) eigenvalue $\lambda_1(\Omega)$. On the other hand,  the maximum in \eqref{localto} is achieved by the solution \(\mathtt{v} \in H_0^1(\Omega)\) (torsion function) to the problem
\begin{equation}
    \label{torl}
\begin{cases}
-\Delta v = 1 & \text{in } \Omega, \\
v = 0 & \text{on } \partial\Omega,
\end{cases}
\end{equation}
and it holds
\[
T(\Omega) = \int_\Omega \mathtt{v}(x) \, dx.
\]

We recall that, among sets with given measure, the ball minimizes $\lambda_1(\Omega)$, as stated by the Lord Rayleigh conjecture, firstly proven by Faber and Krahn (\cite{Fa,Kr}), and maximizes $T(\Omega)$, as stated by the Saint-Venant conjecture, firstly proven by P\'olya (\cite{Po}). However,
in \cite{PS} P\'olya and Szeg\H o stated the stronger conjecture that
among sets with given torsional rigidity, the ball minimizes the principal frequency. A proof of this conjecture was firstly given by Kohler-Jobin in \cite{KJ75,KJ78I} by using a new rearrangement technique known as ``transplantation \`a integrales de Dirichlet \'egales''. Using such a technique, given a smooth positive function $u\in H_0^1(\Omega)$, it is possible to construct a ball $B$ such that $T(B)\le T(\Omega)$ and a radially symmetric decreasing function $\tilde u\in H^1_0(B)$ such that
\begin{equation*}
\int_{B} |D\tilde u|^2 \di x=\int_{\Omega} |D u|^2 \di x\qquad\text{and}\qquad\int_{B} \tilde u^2 \di x\ge \int_{\Omega}  u^2 \di x.
\end{equation*}
Then, P\'olya-Szeg\H o conjecture easily follows and the case of equality can be characterized.

It is worth to point out that the main ingredients used to construct $B$ and $\tilde u$ are the following:
\begin{itemize}
\item[(\textit{i})]for a fixed $u\in H^1_0(B)$, one considers a ``modified'' torsional rigidity on a class of functions in the form $\varphi(u(x))$;
\item[(\textit{ii})] in order to prove the inequality between the $L^2$-norm of $u$ and $\tilde u$ one uses the fact that if \(\mathtt{v} \in H_0^1(\Omega)\) is the torsion function in $\Omega$, that is, it solves \eqref{torl}, then \((\mathtt{v}-t)^+\), $0<t<\max \mathtt{v}$, is the torsion function in $\Omega_t=\{x\in\Omega : \mathtt{v}(x)>t\}$.
\end{itemize}

The approach described above has been extended to various situation, for example, in \cite{Br}, where the first eigenvalue of the $p$-Laplacian and the $p$-torsional rigidity are considered, or in \cite{HL}, where the Gaussian principal frequency and the Gaussian torsional rigidity are considered. A natural question to ask is whether a suitable version of P\'olya-Szeg\H o conjecture holds true in a nonlocal setting.

The fractional Laplacian \((-\Delta)^s\) with \(0 < s < 1\) is a fundamental example of a nonlocal operator, appearing in many areas such as anomalous diffusion, probability, and geometric analysis (see, e.g., \cite{DPV,R,FRR} and the references therein).   
For a sufficiently regular function \(\phi \colon \mathbb{R}^N \to \mathbb{R}\), decaying at infinity, it is defined by  
\begin{equation}\label{fraclapl}
(-\Delta)^s \phi(x) := \gamma(N,s) \, \mathrm{P.V.} \int_{\mathbb{R}^N} \frac{\phi(x) - \phi(y)}{|x-y|^{N+2s}} \, \di y,
\end{equation}
where
\begin{equation}\label{gamma} 
\gamma(N,s)= \left( \int_{\mathbb{R}^N} \frac{1 - \cos(\zeta)}{|\zeta|^{\,N+2s}} \, d\zeta \right)^{-1}= \frac{2^{2s} \, s \, \Gamma\!\left( \frac{N+2s}{2} \right)}{\pi^{\frac N 2} \, \Gamma(1-s)}
\end{equation}
and P.V. stands for the principal value.

When $\Omega\subset \R^N$ is a bounded, open set having Lipschitz boundary, the first Dirichlet eigenvalue of the fractional Laplacian \(\lambda_1(\Omega)\), where for the sake of simplicity the dependence on the parameter $s$ is not explicitly denoted, is defined as the smallest value $\lambda$ so that the problem
\begin{equation} \label{P}
\begin{cases}
(-\Delta)^s u = \lambda u & \text{in } \Omega, \\
u = 0 & \text{in } \mathbb{R}^N \setminus \Omega,
\end{cases}
\end{equation}
has a non-trivial solution in \(X_0^s(\Omega)\), the fractional Sobolev space with zero boundary condition outside \(\Omega\) (see \Cref{sec_lapl} for details). It is well-known that $\lambda_1(\Omega)$ admits the following variational characterization
\begin{equation*}
\lambda_1(\Omega) = \min_{\xi \in X_0^s(\Omega) \setminus \{0\}} \dfrac{\dfrac{\gamma(N,s)}{2}\displaystyle\inte_{\mathbb{R}^N\times \R^N} \frac{|\xi(x) - \xi(y)|^2}{|x - y|^{N + 2s}} \, \di x \di y }
{\displaystyle\int_{\Omega} \xi^2 \di x}.
\end{equation*}

On the other hand, the fractional torsional rigidity $T(\Omega)$ of $\Omega$ is defined as
\begin{equation}\label{tor}
T(\Omega)=\max_{\eta \in X_0^s(\Omega)\setminus \{0\}} \dfrac{\displaystyle\left(\int_{\Omega} |\eta|\di x\right)^2}{\dfrac{\gamma(N,s)}{2}\displaystyle\inte_{\mathbb{R}^N\times \R^N} \frac{|\eta(x) - \eta(y)|^2}{|x - y|^{N + 2s}} \, \di x \di y }
.
\end{equation}
This maximum is attained at the unique function \(\mathtt{v} \in X_0^s(\Omega)\), known as fractional torsion function, which solves the fractional torsion problem
\[
\begin{cases}
(-\Delta)^s v = 1 & \text{in } \Omega, \\
v = 0 & \text{in } \mathbb{R}^N \setminus \Omega.
\end{cases}
\]
We immediately get that
\[
T(\Omega) = \int_\Omega \mathtt{v}(x) \, dx.
\]

Our aim is to prove the following fractional Kohler-Jobin inequality, stating that, among sets with fixed torsional rigidity, the ball has the smallest eigenvalue, i.e.,
\begin{equation}
    \label{KJ_ineq}
\lambda_1(\Omega) \geq \lambda_1(B_R)\qquad\text{where}\ \ B_R\ \ \text{is a ball with radius $R$ s.t.} \ \ T(B_R) = T(\Omega).
\end{equation}
Our initial plan was to follow the strategy developed by Kohler-Jobin in \cite{KJ75, KJ78I}, but, attempting to adapt this approach to the nonlocal case, the main obstacle we faced was that both points (\textit{i}) and (\textit{ii}) above do not seem  to have a natural counterpart in the nonlocal setting. 
More precisely, on one side the use of a function in the form $\varphi(u(x))$ in the nonlocal energy  appearing in the definition of fractional torsional rigidity is not obvious; on the other side the property that if \(\mathtt{v}\) solves \eqref{tor}, then \((\mathtt{v}-t)^+\), $0<t<\max \mathtt{v}$, is the torsion function in $\Omega_t=\{x\in\Omega : \mathtt{v}(x)>t\}$ is false in the nonlocal context.

Therefore, we have used a different approach which is based on the fact that the fractional torsional rigidity $T(\Omega)$ can be seen as a particular case of a ``generalized fractional torsional rigidity'', defined, for $\alpha \in \R$, as
\begin{equation}
    \label{torgen_max_intro}
Q(\alpha, \Omega) = \max_{\psi \in X_0^s(\Omega)} \left\{ - \frac{\gamma(N,s)}{2}\inte_{\mathbb{R}^N\times \R^N} \frac{|\psi(x) - \psi(y)|^2}{|x - y|^{N + 2s}} \, \di x \di y + \alpha \int_\Omega |\psi(x)|^2 \, \di x + 2 \int_\Omega \psi(x) \, \di x \right\},
\end{equation}
which has been firstly introduced in \cite{Ba} when the local case is considered.
For any \(\alpha \in (-\infty, \lambda_1(\Omega))\),
the maximum in \eqref{torgen_max_intro} is attained at the function \(\mathtt{w}\) (generalized torsion function), which is the solution to the problem
\begin{equation*}
\begin{cases}
(-\Delta)^s w = \alpha w + 1 & \text{in } \Omega, \\
w = 0 & \text{in } \mathbb{R}^N \setminus \Omega,
\end{cases}
\end{equation*}
and it is immediate to observe that $Q(0,\Omega)=T(\Omega)$.

In this paper we will prove that, for any \(\alpha \in (-\infty, \lambda_1(\Omega))\) and for any bounded open set \(\Omega \subset \mathbb{R}^N\) with Lipschitz boundary, the following inequality holds true:
\begin{equation}\label{tor-alpha}
\lambda_1(\Omega) \geq \lambda_1(B_{R(\alpha)})\quad\text{where}\ \ B_{R(\alpha)}\ \ \text{is a ball with radius $R(\alpha)$ s.t.} \ Q(\alpha, \Omega) = Q(\alpha, B_{R(\alpha)}).
\end{equation}
Clearly, when \(\alpha = 0\), the inequality reduces to \eqref{KJ_ineq}. The proof is based on the fact that the mapping $\alpha \mapsto R(\alpha)$ in \eqref{tor-alpha} is decreasing and the full statement \eqref{tor-alpha} follows taking the limit as $\alpha\rightarrow\lambda_1(\Omega)$. 
Let us observe that, in the local context, a similar approach has been adopted in \cite{KJ77I,KJ78I,KJ78II}, where, using different techniques, the counterpart of \eqref{tor-alpha} is proven. Unfortunately, since \eqref{tor-alpha} is obtained via a limit procedure, it seems that the method does not give a characterization of the equality case (see \cite[Prop.4.1]{SVV} for the study of the equality case in a nonlocal problem).

As we have already said, in order to study the properties of $Q(\alpha, \Omega)$, the techniques employed in the local context do not seem to be appropriate, and the main ingredient in our proof is a comparison result between the generalized torsion functions in $\Omega$ and in $B_{R(\alpha)}$ in terms of mass concentration estimates. Such a result is based on symmetrization techniques introduced in \cite{FV} (see also \cite{BDFV}, \cite{FPV}) and it can be seen as the natural counterpart of similar ``local'' results contained in \cite{C1} and based on the well-known symmetrization techniques developed by Talenti \cite{T}.

As already observed in \cite{C2}, a comparison result of the type described above can be used in order to prove a so-called Payne-Rayner inequality (see \cite{PR1,PR2,KJ77,KJ81}) which, in the original formulation, 
provides a sharp estimate for the $L^{2}$ norm of a first Dirichlet-Laplacian eigenfunction   in terms of its $L^{1}$ norm. This kind of reverse H\"older inequality was generalized in \cite{C2,AFT}, where the authors showed that the $L^{q}$ norm of an eigenfunction of a linear, or even nonlinear, operator in divergence form can be sharply estimated by its $L^{p}$ norms whenever $q \ge p\ge 1$ (see also \cite{BCT} in the case of Neumann boundary conditions). 
In this paper, we will prove that, for any eigenfunction $\mathtt u_1$ corresponding to $\lambda_1(\Omega)$ and for any $1<q\le +\infty$, the following reverse H\"older inequality holds true:
\begin{equation*}
\|\mathtt u_1\|_{L^q(\Omega)} \leq C\lambda_1(\Omega)^{\frac N{2s}\left(1-\frac1q\right)} \|\mathtt u_1\|_{L^1(\Omega)},
\end{equation*}
where the value of the positive $C=C(N,s,q)$ is explicitly given. Unfortunately, our techniques do not seem to work in order to prove a more general result such as a $p-q$ reverse H\"older inequality ($q \ge p\ge 1$) in the nonlocal setting.

The paper is structured as follows. 
In \Cref{sec_notation}, we introduce notation and preliminaries. \Cref{sec_lapl} is devoted to the fractional Laplacian spectral problem and the fractional torsional rigidity, while \Cref{sec_torsion} is dedicated to a thorough analysis of the generalized torsion and its properties. We then present a key comparison result, that is crucial for then deriving both the Kohler-Jobin (see \Cref{sec_compar}) and the reverse H\"older (see \Cref{pay_sec}) inequalities.

%%%%%%%%%%%%%%%%%%%%%%%%%%%%%%

\section{Notation and preliminaries}
\label{sec_notation}

From now on, we denote by $B_r(x_0)$ the open  ball in $\mathbb{R}^N$, centered at $x_0$, with radius $r$ and we write $B_r=B_r(0)$. $B_r^c(x_0)$ stands for the complement of the ball $B_r(x_0)$ and $\omega_N$ for the measure of the unitary ball, that is
$$
\omega_N =\frac{\pi^{\frac N 2}}{\Gamma\left(\frac N 2 +1\right)}.
$$
Furthermore, for any set $E\subseteq \R^N$, we  denote by $E^{\sharp}$ the ball in $\mathbb{R}^{N}$, centered at the origin, with the same Lebesgue measure as $E$ ($E^{\sharp}=\R^N$ if $|E|=+\infty$). 

In this section, we recall the definition of decreasing rearrangement and some of its properties, which will be useful in the following. For a more exhaustive treatment of the argument we refer the interested reader, for example, to \cite{CR,HLP,Ka,Ke}.

Let us consider a real measurable function $f$ on an open set $\Omega\subset\R^N$ and, for any $t\geq 0$, the super-level set 
\[\Omega_f^t=\left\{  x\in\Omega:\left\vert f\left(
x\right)  \right\vert >t\right\}.\]
We define the \emph{distribution function} $\mu_{f}$ of $f$ as follows
\begin{equation*}
\mu_{f}(t)  =\left\vert \Omega_f^t\right\vert \qquad\text{for every }t
\ge0,
\end{equation*}
and we assume that $\mu_f(t)<+\infty$ for every $t > 0$.
By definition,  $\mu_f(\cdot)$ is a right-continuous function, decreasing from $\mu_f (0)=|\supp(f)|$ to $\mu_f(+\infty)=0$ as $t$ increases from 0 to $+\infty$. It presents a discontinuity at every value $t$ which is assumed by $|f|$ on a set of positive measure, and, for such a value of $t$, we have
\[
\mu_f(t^-)-\mu_f(t)=|\{x\in\Omega:\left\vert f\left(
x\right)  \right\vert =t \}|.
\]
For every $t\geq 0$, we set
\begin{equation*}
{r_f(t)=\left(\frac{\mu_f(t)}{\omega_N}\right)^\frac 1N}  \quad \mbox{and} \quad r_f(t^-)=\left(\frac{\mu_f(t^-)}{\omega_N}\right)^\frac 1N.
\end{equation*}
It is clear that $(\Omega_f^t)^\sharp=B_{r_f(t)}$ and that $r_f(t)$ is also a right-continuous function. 

\noindent The \emph{(one dimensional) decreasing rearrangement} $f^\ast$ of $f$ is defined as follows
\begin{equation*}
    \label{f^*}
f^{\ast}\left(  \sigma\right)  =\sup\left\{ t\geq0:\mu_{f}\left(  t\right)
>\sigma\right\}\qquad\sigma\in\left[0,+\infty \right[,
\end{equation*}
that is, $f^*$ is the distribution function of $\mu_f$.
We stress that, if $\mu_f$ is strictly decreasing, then $f^*$ extends to the whole half-line $[0,+\infty[$ the inverse function of $\mu_f$. In the general case, we have that $f^*(\mu_f(t))\le t$, for $t\in[0,+\infty[$, and $\mu_f(f^*(\sigma))\le \sigma$, for $\sigma\in[0,+\infty[$. We also observe that, if $\mu_f(t)$ has a jump, i.e., $\mu_f(t)<\mu_f(t^-)$ for some $t$, then $f^*(\sigma)$ has a flat zone, i.e., $f^*(\sigma)=t$ for every $\sigma\in[\mu_f(t),\mu_f(t^-)]$ (see Figure 1). Similarly, if $\mu_f(t)$ has a flat zone, then $f^*(s)$ has a jump.

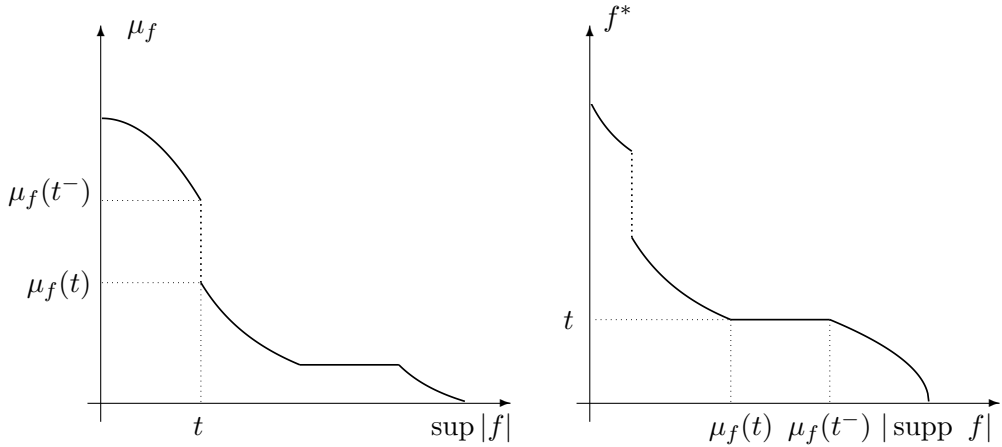
\begin{figure}[ht]\label{figura}
{\begin{picture}(450,200)
\put(29.5,-2.5){\begin{tikzpicture}[scale=.8]
\begin{axis}[samples=100,xmin=0,xmax=2.5,
      ymin=0,ymax=2.5,
    axis line style={draw=none},
      tick style={draw=none}, xticklabels={,,},
      yticklabels={,,}]
    \addplot[domain = 1.2:1.8, thick,dotted](0.6,x); 
    \addplot[domain = .33:1.2,dotted](0.6,x); 
    \addplot[domain = 0:.6,dotted](x,1.8); 
    \addplot[domain = 0:.6,dotted](x,1.2); 
    \addplot[domain = 0:{0.6}, thick](x,{2.4-5*(x^2)/3}); 
   \addplot[domain = .6:1.2, thick](x,{18/(25*x)}); 
    \addplot[domain = 1.2:1.8, thick](x,3/5); 
   \addplot[domain = 1.8:2.2, thick](x,{3/(10*x-13)}); 
\end{axis}
\end{tikzpicture}}
\put(192.5,-2.5){\begin{tikzpicture}[scale=.8]
\begin{axis}[samples=100,xmin=0,xmax=2.5,
      ymin=0,ymax=2.5,
    axis line style={draw=none},
      tick style={draw=none}, xticklabels={,,},
      yticklabels={,,}]
    \addplot[domain = 1.53:2.15,thick,dotted](0.6,x); 
    \addplot[domain = 0.33:.93,dotted](1.2,x); 
    \addplot[domain = 0.33:.93,dotted](1.8,x); 
    \addplot[domain = 0.38:1.2,dotted](x,.93); 
   \addplot[domain = .6:1.2, thick](x,{.33+18/(25*x)}); 
   \addplot[domain = 0.2:{0.6}, thick](x,{1.66+.3/x}); 
    \addplot[domain = 1.8:{2.4}, thick](x,{.33+sqrt((2.4-x)*3/5)}); 
    \addplot[domain = 1.2:1.8, thick](x,0.93); 
\end{axis}
\end{tikzpicture}}
\put(30,20){\vector(1,0){160}}
\put(35,13){\vector(0,1){150}}
\put(215,20){\vector(1,0){160}}
\put(220,13){\vector(0,1){150}}
\put(225,163){$f^*$}
\put(45,160){$\mu_f$}
\put(70,8){$t$}
\put(7,62){$\mu_f(t)$}
\put(0,97){$\mu_f(t^-)$}
\put(265,8){$\mu_f(t)$}
\put(295,8){$\mu_f(t^-)$}
\put(330,8){$|\supp\ f|$}
\put(160,8){$\sup|f|$}
\put(210,48){$t$}
\end{picture}}
\caption{On the left, a distribution function which presents a discontinuity and a flat zone; on the right, the corresponding decreasing rearrangement.}
\end{figure}

\noindent If $\Omega$ has a finite measure, we can also define the 
\emph{(one dimensional) increasing rearrangement} $f_\ast$ of $f$, that is
\[
f_{\ast}\left(\sigma\right)  =f^\ast(|\Omega|-\sigma),\qquad\sigma\in\left(  0,\left\vert \Omega\right\vert \right).
\]
We call the \emph{radially decreasing rearrangement} (or \emph{Schwarz decreasing rearrangement}) $f^\sharp$ of $f$ the function defined as
\[
f^{\sharp}\left(  x\right)  =f^{\ast}(\omega_{N}\left\vert x\right\vert
^{N}),\qquad x\in\Omega^{\sharp},
\]
while we call the \emph{radially increasing rearrangement} $f_\sharp$ of $f$ the
function
\[
f_{\sharp}\left(  x\right)  =f_{\ast}(\omega_{N}\left\vert x\right\vert
^{N}),\qquad x\in\Omega^{\sharp}.
\]
From the definitions, we immediately deduce that $f^*$, $f_*$, $f^\sharp$ and $f_\sharp$ have the same distribution function as $f$. As a consequence, by the layer cake formula, rearrangements preserve the $L^{p}$ norms, that is
\begin{equation*}
\|f\|_{L^{p}(\Omega)}=\|f^{\ast}\|_{L^{p}(0,|\Omega|)}=\|f^{\sharp}\|_{L^{p}(\Omega^{\sharp})}, \qquad 1 \le p \le +\infty.
\end{equation*}
Furthermore, for any couple of measurable functions $f$ and $g$, the classical Hardy-Littlewood inequalities holds true
\begin{equation*}
\int_{\Omega}\vert f(x)\,  g(x)  \vert \di x\leq\int_{0}^{\left\vert \Omega\right\vert}f^{\ast}(\sigma)\,  g^{\ast}(\sigma)  \di\sigma=\int_{\Omega^{\sharp}}f^{\sharp}(x)\,g^{\sharp}(x)\,\di x\,,
\end{equation*}
and
\begin{equation*}
\int_{\Omega^{\sharp}}f^{\sharp}(x)\,g_{\sharp}(x)\,\di x\,=\int_{0}^{\left\vert \Omega\right\vert}f^{\ast}(\sigma)\,  g_{\ast}(\sigma)  \di \sigma \le \int_{\Omega}\vert f(x)\,  g(x)  \vert \di x.
\end{equation*}

Since we will deal with integrals of solutions to nonlocal problems, the following definition will play a fundamental role.

\begin{definition}
 Let $f,g\in L^1_{\mathrm{loc}}(\R^N)$.
 We say that $f$ is less concentrated than $g$, and we write $f \prec g$, if for every $\sigma>0$ we have
 $$
 \int_0^\sigma u^\ast(t)\, \di t \le \int_0^\sigma v^\ast(t)\, \di t,
 $$
 or, equivalently, for every $r>0$,
 $$
 \int_{B_r}f^\sharp(x)\, \di x\le \int_{B_r}g^\sharp(x)\, \di x.
 $$
\end{definition}
\noindent Clearly, this definition can be adapted to functions defined in an open subset $\Omega$ of $\R^N$, by extending the functions to zero outside $\Omega$. 
The partial order relationship $\prec$ is called comparison of mass concentrations and it satisfies some nice properties (see, for instance, \cite{ATL}).

\begin{proposition}\label{Propconves}
Let $f,g \in L^1(\Omega)$ be two nonnegative functions. Then, the following statements are equivalent:
\begin{itemize}
\item[(a)] $f \prec g$;
\item[(b)] for all nonnegative $\varphi\in L^{\infty}(\Omega)$
\begin{equation}\label{B}
\int_\Omega f(x)\varphi(x)\, \di x\leq \int_0^{|\Omega|} g^\ast(r)\varphi^\ast(r)\, \di r = \int_{\Omega^\sharp}g^\sharp(x)\varphi^\sharp(x)\, \di x;
\end{equation}
\item[(c)] for all convex, nonnegative, Lipschitz function $\Phi$, such that $\Phi(0)=0,$
$$\int_{\Omega}\Phi(f(x))\,\di x\leq \int_{\Omega^\sharp}\Phi(g(x))\,\di x.$$
\end{itemize}
\end{proposition}

\noindent From Proposition \ref{Propconves} we immediately deduce that, if $f \prec g$, then
$$
||f||_{L^p(\Omega)}\le ||g||_{L^p(\Omega)}, \qquad 1 \le p \le +\infty.
$$
Moreover, if $f,\,g\in L^{p}(\Omega)$ with $p>1$, inequality \eqref{B} holds true for all nonnegative $\varphi\in L^{p^{\prime}}(\Omega)$, where $\frac 1 p + \frac{1}{p'}=1$.

We end this section by recalling the celebrated P\'olya-Szeg\H o principle, stating that the radially decreasing rearrangement $f^\sharp$ of a Sobolev function $f$ is a Sobolev function and its energy does not exceed the energy of $f$.
\begin{proposition}\label{PS}
Let $1 \le p <\infty$ and let $f \in W^{1,p}(\R^N)$. Then $f^\sharp \in W^{1,p}(\R^N)$ and the following inequality holds true
$$
\int_{\R^N} |\nabla f|^p\, \di x \ge \int_{\R^N} |\nabla f^\sharp|^p \, \di x.
$$
\end{proposition}

%%%%%%%%%%%%%%%%%%%%%%%%

\section{Fractional Laplacian: The Eigenvalue Problem and the Torsional Rigidity} \label{sec_lapl}
Let $\Omega\subset\RR$ be an open set and take $s \in (0,1)$. As already stated in the Introduction, we define the fractional Laplacian of a smooth and decaying real function $\phi$ on $\R^N$ by \eqref{fraclapl}.
The choice of $\gamma(N,s)$ in  in \eqref{gamma} ensures that
\((-\Delta)^s u\) converges to the classical Laplacian \(-\Delta u\) as \(s \to 1^{-}\) (see \cite{DPV}).

Denoted by $[\phi]_{H(\mathbb{R}^N)}$ the fractional Gagliardo seminorm of $\phi$, that is
\begin{equation*}
    \label{gagl_semi}
[\phi]_{H^{s}(\RR)}=\left(\dfrac{\gamma(N,s)}{2}\inte_{ \R^N\times\R^N}\frac{|\phi(x)-\phi(y)|^{2}}{|x-y|^{N+2s}}\,\di x\di y\right)^{\frac 1 2},
\end{equation*}
the Sobolev space $H^s(\mathbb{R}^N)$ is defined as
\[
H^s(\mathbb{R}^N) = \left\{\, \phi \in L^2(\mathbb{R}^N): \> \ [\phi]_{H^s(\mathbb{R}^N)} < +\infty \,\right\},
\]
 equipped with the norm
\begin{equation*}\label{norm}
\|\phi\|_{H(\mathbb{R}^N)} =
\left( \, \|\phi\|_{L^2(\mathbb{R}^N)}^2 + [\phi]_{H^s(\mathbb{R}^N)}^2 \, \right)^{\frac 1 2}.
\end{equation*}

Since we are interested in Dirichlet problems defined in bounded domains,  we  consider the space $X_0^s(\Omega)$, defined as 
$$
X_0^s(\Omega)=\left\{\phi \in H^s(\R^N): \> \phi=0 \mbox{ a.e. in } \R^N \setminus \Omega\right\}.
$$
When $\Omega$ is a  bounded, open set with Lipschitz boundary, it can be proven that (see \cite[Proposition B.1]{BPS}) $X_0^s(\Omega)$ coincides with the completion of $C_{0}^{\infty}(\Omega)$ with respect to the seminorm $ [\cdot]_{H^{s}(\R^{N})}$. 
\\
A consequence of fractional Poincar\'e inequality (see \cite[Lemma 2.4]{BLP}) is that we can equip the space $X_0^s(\Omega)$ with the Gagliardo seminorm
\[
\|\phi\|_{X_{0}^{s}(\Omega)}=[\phi]_{H^{s}(\R^{N})}=\left(\dfrac{\gamma(N,s)}{2}\inte_{\R^{N}\times \R^N}\frac{|\phi(x)-\phi(y)|^{2}}{|x-y|^{N+2s}}\,\di x\di y\right)^{\frac 1 2}.
\]

\noindent From the definition of $X_{0}^{s}(\Omega)$ it easily follows that for each $\phi\in X_{0}^{s}(\Omega)$
\begin{equation*}\label{norm2}
||\phi||_{X_0^s(\Omega)} = \left(\dfrac{\gamma(N,s)}{2}\inte_{Q}\frac{|\phi(x)-\phi(y)|^2}{|x-y|^{N+2s}}\di x\di y\right)^{\frac 1 2}
\end{equation*}
where $Q=\R^{2N}\setminus\left(\Omega^c \times \Omega^c\right)$ and $\Omega^c=\R^N\setminus\Omega$.\\
Then we consider the \emph{restricted} fractional Laplacian $(-\Delta|_{\Omega})_{rest}$ on $\Omega$, defined by duality on the space $X_{0}^{s}(\Omega)$. Since there will be no matter of confusion, we shall keep the classical notation $(-\Delta)^{s}$ for such operator. Moreover, denoted by $X^{-s}(\Omega)$ the dual of $X_0^s(\Omega)$, the operator $$(-\Delta)^s: X_0^s(\Omega) \to X^{-s}(\Omega)$$ is continuous. 
Finally, we recall that the following fractional Sobolev embedding holds true (see for instance \cite{BLP}).

\begin{theorem}\label{sobolev}
Let $s \in (0,1)$ and $N>2s$. There exists a positive constant $\mathcal{S}(N,s)$ such that, for any measurable and compactly supported function $\phi:\R^N \to \R$, it holds
\begin{equation*}
||\phi||_{L^{2_s^\ast}(\R^N)}^2 \le \mathcal{S}(N,s)\inte_{\R^N\times\R^N} \frac{|\phi(x)-\phi(y)|^2}{|x-y|^{N+2s}}\,\di x\di y,
\end{equation*}
where
$$
2_s^\ast=\frac{2N}{N-2s}
$$
is the critical Sobolev exponent. In particular, if $\phi \in X_0^s(\Omega)$, we have
\begin{equation} \label{emb}
||\phi||_{L^{2_s^\ast}(\Omega)}^2 \le \mathcal{S}(N,s)\inte_{\R^N \times \R^N}\frac{|\phi(x)-\phi(y)|^2}{|x-y|^{N+2s}}\,\di x\di y,
\end{equation}
that is the space $X_0^s(\Omega) $ is continuously embedded in $L^{2_s^\ast}(\Omega).$ Moreover, $X_0^s(\Omega)$ is compactly embedded in $L^q(\Omega)$, for every $1 \le q < 2_s^\ast$.
\end{theorem}
For more details on fractional Sobolev spaces and nonlocal operators we refer the interested reader  to \cite{FKV,R}.

Now, we recall that the radially decreasing rearrangement of a Sobolev function is a Sobolev function and that the fractional Gagliardo seminorm does not increase under rearrangement. The following proposition can be seen as the nonlocal counterpart of the P\'olya-Szeg\H o principle recalled in Proposition \ref{PS} (see \cite[Theorem 9.2]{AL}, see also \cite[Theorem A.1]{FS}).

\begin{proposition}\label{prop_polya}
For any $\phi \in H^{s}(\mathbb{R}^N)$, the following inequality holds true
\begin{equation}
    \label{eq_polya}
\inte_{\mathbb{R}^N \times \mathbb{R}^N} 
\frac{|\phi(x) - \phi(y)|^2}{|x-y|^{N+2s}} \, dx \, dy
\;\ge\;
\inte_{\mathbb{R}^N \times \mathbb{R}^N} 
\frac{|\phi^\sharp(x) - \phi^\sharp(y)|^2}{|x-y|^{N+2s}} \, dx \, dy.
\end{equation}
The equality sign in \eqref{eq_polya} is achieved if and only if $\phi$ is proportional to a (translation of a) radially symmetric, decreasing function.
\end{proposition}

\subsection{The Fractional Eigenvalue Problem}

Let $\Omega\subset\R^N$, $N\geq 2$, be a bounded open set having Lipschitz boundary.
We consider the nonlocal eigenvalue problem \eqref{P},
whose weak formulation reads as
\begin{equation}\label{P1}
\left\{
\begin{array}{ll}
\dfrac{\gamma(N,s)}{2}\displaystyle\inte_{\R^N \times \R^N} \frac{\left(u(x)-u(y)\right)\left(\varphi(x)-\varphi(y)\right)}{|x-y|^{N+2s}}\,\di x\di y =\lambda \displaystyle\int_\Omega u(x)\varphi(x)\, \di x, \qquad \varphi \in X_0^s(\Omega),
\\ \\
u \in X_0^s(\Omega).
\end{array}
\right.
\end{equation}
We recall that $\lambda\in \R$ is called an eigenvalue if there exists a nontrivial solution $u \in X_0^s(\Omega)$ to \eqref{P1} and, in this case, any solution is called an eigenfunction corresponding to the eigenvalue $\lambda$. It is well-known (see, for example, \cite{SV13}) that:
\smallskip

\noindent 1) problem \eqref{P1} admits the smallest eigenvalue $\lambda_1(\Omega)$ which is positive and that can be characterized as follow
\begin{equation}\label{eigen}
\lambda_1(\Omega)=\min_{\xi\in X_0^s(\Omega)\setminus\{0\}}
\frac{ [\xi]^2_{H^s(\R^N)} }{||\xi||_{L^2(\Omega)}^2};
\end{equation}
\smallskip

\noindent 2) there exists a positive function $\mathtt u_1\in X_0^s(\Omega)$, which is an eigenfunction corresponding to $\lambda_1(\Omega)$, attaining the minimum in \eqref{eigen};
\smallskip

\noindent 3) $\lambda_1(\Omega)$ is simple, that is, if $\mathtt u \in X_0^s(\Omega)$ is a solution to the following equation
$$
\dfrac{\gamma(N,s)}{2}\displaystyle\inte_{\R^N \times \R^N} \frac{\left(u(x)-u(y)\right)\left(\varphi(x)-\varphi(y)\right)}{|x-y|^{N+2s}}\,\di x\di y =\lambda_1(\Omega) \displaystyle\int_\Omega u(x)\varphi(x)\, \di x, \qquad \varphi \in X_0^s(\Omega)
$$
then $\mathtt u=\alpha \mathtt u_1$, with $\alpha \in \R$;
\smallskip

\noindent 4)  $\lambda_1(\Omega)$ is monotone decreasing with respect to the inclusion of sets, that is, if $\Omega' \subset \Omega$, then $\lambda_1(\Omega')\ge \lambda_1(\Omega)$. Moreover, it scales under dilation as follows:
\begin{equation}\label{eigen-scal}
\lambda_1(t\Omega)=t^{-2s}\lambda_1(\Omega), \quad t>0.
\end{equation}

Using the Sobolev inequality contained in Theorem \ref{sobolev}, we can immediately derive the existence of a positive constant $C=C(N,s)$ such that
$$
\lambda_1(\Omega) \ge C\,|\Omega|^{-\frac{2s}{N}}.
$$

\begin{remark}\label{regularity}
By standard arguments, we can show that any eigenfunction is bounded and smooth inside $\Omega$. We start by knowing $\mathtt u_1\in L^{p}$ with $p=2^{\ast}_{s}<N/2s$ by the fractional Sobolev embedding \eqref{emb}. Then we use \cite[Th. 3.2]{FV} with $f=\lambda \mathtt u_1$ in order to get $\mathtt u_1\in L^{q}$ with $q=2N/(N-6s)>2^{\ast}_{s}$. Bootstrapping, after a finite number $k$ of steps we have that $\mathtt u_1\in L^{q_{k}}$ with $q_{k}>N/2s$. Thus \cite[Th. 3.2]{FV} again gives $\mathtt u_1\in L^{\infty}(\Omega)$. Now using \cite[Theorem 2.4.1, Proposition 2.4.4]{FRR} or \cite[Theorem 1.1]{RSStable} we have that $\mathtt u_{1}\in C^{\alpha}_{loc}(\Omega)$ for some $\alpha=\alpha(s)$. Hence, $f=\lambda\mathtt {u}_{1}\in C^{\alpha}_{loc}(\Omega)$ and the Schauder regularity gives $\mathtt {u}_{1}\in C^{\alpha+2s}_{loc}(\Omega)$ when $\alpha+2s\not\in \mathbb{N}$. Bootstrapping, after a finite number of steps $\mathtt u_1\in C^{\infty}(\Omega)$.
\end{remark}

The fractional Faber-Krahn inequality stated in the following theorem says that the optimal value of the constant
$C(N,s)$ is attained when $\Omega$ is a ball. 
To the best of our knowledge, the proof of the Faber-Krahn inequality can be found in \cite[Theorem 3.5]{BLP}. Nonetheless, it essentially builds upon the P\'olya-Szeg\H o principle for the Gagliardo seminorm stated in \Cref{prop_polya}.
\begin{proposition}\label{FK_prop}
Let $\Omega\subset\R^N$ be a bounded, open set having Lipschitz boundary. Then
\begin{equation}\label{FK}
\lambda_1(\Omega) \geq \lambda_1(\Omega^\sharp)\qquad\mbox{ where } \Omega^\sharp \mbox{ is the ball (centered at the origin) s.t.} \ \ |\Omega^\sharp| = |\Omega|.
\end{equation}
Equality holds if and only if \(\Omega\) is a ball.
\end{proposition}

\begin{remark}
Unlike the first eigenvalue, for the second eigenvalue of the fractional Dirichlet Laplacian  an optimal shape under volume constraints is not known. For example, in  \cite{BP} the authors show that a minimizing sequence is given by two disjoint balls each of volume \(|\Omega|/2\) whose mutual distance tends to infinity.
\end{remark}

We end this subsection by recalling the following result on eigenvalues of balls contained in \cite{D}.

\begin{proposition}
Let $\lambda_\ast$ be the smallest number such that there exists an eigenfunction $\phi_\ast$ of the fractional Dirichlet-Laplacian in the unitary ball $B_1 $ in $ \R^N$ which is antisymmetric, i.e. $\phi_\ast(-x)=-\phi_\ast(x)$, and has eigenvalue $\lambda_\ast$.
Then 
$$\lambda_\ast=\lambda_{1,N+2}(B_1),$$
where $\lambda_{1,N+2}(B_1)$ is the first eigenvalue of the unitary ball in $\R^{N+2}$.
\end{proposition}

\begin{remark}\label{dydarem}
As a consequence, we immediately get that the first eigenvalue of the fractional Dirichlet-Laplacian on balls is increasing with respect to the dimension $N$.
\end{remark}

\subsection{The Fractional Torsional Rigidity}
The fractional torsional rigidity of $\Omega$ has been defined in \eqref{tor}. 
It can be easily seen that the maximum in \eqref{tor} is attained at a unique function \(\mathtt{v} \in H_0^s(\Omega)\), which solves the fractional torsion problem
\begin{equation}
    \label{tor_pb}
\begin{cases}
(-\Delta)^s v = 1 & \text{in } \Omega, \\
v = 0 & \text{in } \mathbb{R}^N \setminus \Omega,
\end{cases}
\end{equation}
whose weak formulation reads as 
\[
\frac{\gamma(N,s)}{2}\inte_{\R^N\times \R^N}
\frac{\bigl( v(x)- v(y)\bigr)\bigl(\varphi(x)-\varphi(y)\bigr)}
{|x-y|^{N+2s}}\,\di x \di y
=\int_\Omega \varphi(x)\,\di x,
\qquad \varphi\in X_0^s(\Omega).
\]
Obviously, the value of the maximum in \eqref{tor}  can be equivalently expressed as
\[
T(\Omega) = \int_\Omega \mathtt{v}(x) \, \di x.
\]
As for the first eigenvalue, it is easy to verify that the torsional rigidity scales under dilation as 
\begin{equation}\label{tor-scal}
T(t\Omega)=t^{N+2s}T(\Omega), \qquad t>0.
\end{equation}
To the best of our knowledge, the Saint-Venant inequality in the nonlocal setting has not been explicitly stated, and it has so far been established only in the particular context of random walk spaces (see \cite{MT}).  

The proof, similarly to the one of the Faber-Krahn inequality \eqref{FK}, essentially relies on the P\'olya-Szeg\H o principle for the Gagliardo seminorm stated in Proposition \ref{prop_polya}.
\begin{proposition}
Let $\Omega \subset \mathbb R^N$ be a bounded, open set having Lipschitz boundary. Then
\begin{equation}\label{SV}
T(\Omega) \leq T(\Omega^\sharp), \qquad \mbox{ where } \Omega^\sharp \mbox{ is the ball (centered at the origin) s.t. } |\Omega^\sharp| = |\Omega|.
\end{equation}
Equality holds if and only if \(\Omega\) is a ball.
\end{proposition}

We mention here \cite{O}, treating the fractional version of the torsional rigidity on graphs. We also mention that in \cite{CDPPV} symmetry and quantitative stability results for the parallel surface fractional torsion problem have been established.

\vskip 0.5cm

When $\Omega=B_1$, the following explicit expression for the unique solution $\bar{\mathtt v}$ to \eqref{tor_pb} has been provided in \cite{D}:  
\begin{equation}
    \label{tor_sol}
\bar{\mathtt v}(x) =   
\dfrac{\Gamma\left(\frac N2\right)}{4^{s}\,\Gamma(1+s)\,\Gamma\left(\frac{N+2s}{2}\right)} (1-|x|^2)^s_+. 
\end{equation}
We prove the following
\begin{lemma}\label{Lemma3.1}
Let $\bar{\mathtt v}$ be defined as in \eqref{tor_sol}, then 
\begin{equation}
    \label{upp_v}
\bar{\mathtt v}(x) \le 
\begin{cases}
\dfrac{1}{\Gamma(x_0)}  & \text{ if } N=1\\
1 & \text{ if } N\ge 2,
\end{cases}
\end{equation}
where
$$
\Gamma(x_0)=\min_{x \in [1,3]}\Gamma(x).
$$
\end{lemma}
\begin{proof}
Let \(N=1\). The Lagrange's Duplication Formula for the Gamma function (see, for example, \cite{AS}) guaranties that
\[
\Gamma(x)\, \Gamma\Big(x+\frac{1}{2}\Big) = 2^{1-2x} \, \sqrt{\pi} \, \Gamma(2x).
\]
If we apply it by taking \(x = s + \frac{1}{2}\), recalling that \(\Gamma\Big(\frac{1}{2}\Big) = \sqrt{\pi}\), we immediately get
\[
\frac{\Gamma\Big(\frac{1}{2}\Big)}{4^s \, \Gamma(s+1)\, \Gamma\Big(s+\frac{1}{2}\Big)} = \frac{1}{\Gamma(2s+1)}.
\]
Moreover, since \(2s+1 \in [1,3]\), we have \(\Gamma(2s+1) \ge \Gamma(x_0) \simeq 0.8856\), where \(x_0 \simeq 1.4616\) is the minimum point of \(\Gamma\) in the interval \([1,3]\) (see \cite[Chapther 6]{AS} for a comprehensive account). Thus,  
\[
\bar{\mathtt v}(x) \le \frac{1}{\Gamma(x_0)} \simeq 1.1292.
\]

When \(N\ge2\), the bound on $\bar{\mathtt v}$ can be improved using the fact that the Gamma function is log-convex on $(0,+\infty)$ (see, for example \cite{AB}), that is the function
$$g(x)=\log \Gamma(x)$$
is convex on $(0,+\infty)$. Then the function
$$c_N(s)=\log \dfrac{\Gamma\left(\frac N2\right)}{4^{s}\,\Gamma(1+s)\,\Gamma\left(\frac{N+2s}{2}\right)}=
g\left({\mbox{$\frac N2$}}\right)-s\log4-g(1+s)-g\left({\mbox{$\frac N2+s$}}\right)$$
is concave on $[0,1]$. Furthermore, being $g'$ increasing, we obtain
$$c'_N(s)=-\log4-g'(1+s)-g'\left({\mbox{$\frac N2+s$}}\right)\le-\log4-2g'(1),\quad s\in[0,1].$$
Recalling that 
$$g'(1)=\frac{\Gamma'(1)}{\Gamma(1)}=-\gamma\simeq 0.5772$$
where $\gamma$ is the Euler-Mascheroni constant  (see, for instance, \cite{AS})  and taking into account the fact that $\log4\simeq 1.3863$, it follows that
$$c'_N(s)<0,\quad s\in[0,1].$$
On the other hand, $c_N(1)=0$, so $c_N(s)\le0$ for $s\in[0,1]$, that is,
$$\dfrac{\Gamma\left(\frac N2\right)}{4^{s}\,\Gamma(1+s)\,\Gamma\left(\frac{N+2s}{2}\right)}\le1,\quad s\in[0,1].$$
\end{proof}

%%%%%%%%%%%%%%%%%%%%%%%%%%%%

\section{A generalized fractional torsional rigidity}\label{sec_torsion}
For our purposes, we introduce a generalized version of the fractional torsional rigidity, first introduced in \cite{Ba} in the local case. Specifically, for $\alpha\in \R$, we consider (see \eqref{torgen_max_intro})
\begin{equation}
\label{torgen_max}
Q(\alpha, \Omega) = \sup_{\psi \in X_0^s(\Omega)} \left\{ -[\psi]^2_{H^{s}(\mathbb{R}^N)} + \alpha \int_\Omega |\psi(x)|^2 \, \mathrm{d}x + 2 \int_\Omega \psi(x) \, \mathrm{d}x \right\}.
\end{equation}
For any $\alpha\in (-\infty,\lambda_{1}(\Omega))$, the functional in \eqref{torgen_max} is bounded from above since, using  \eqref{eigen} and Young inequality, it holds that, for some positive $C$,
$$
-[\psi]^2_{H^{s}(\mathbb{R}^N)} + \alpha \int_\Omega |\psi(x)|^2 \, \mathrm{d}x + 2 \int_\Omega \psi(x) \, \mathrm{d}x \le 
C |\Omega|.
$$
Via classical arguments of semicontinuity and compactness, 
the maximum in \eqref{torgen_max} is attained at \(\psi = \mathtt{w}\), where \(\mathtt{w}\) is the unique solution to the problem
\begin{equation}
\label{torgen_pb}
\begin{cases}
(-\Delta)^s w = \alpha w + 1 & \text{in } \Omega, \\
w = 0 & \text{in } \mathbb{R}^N \setminus \Omega,
\end{cases}
\end{equation}
whose weak formulation reads as
\begin{equation}
    \label{torgen_weak}
\dfrac{\gamma(N,s)}{2} \inte_{\R^N\times \R^N} \frac{(w(x) - w(y))(\varphi(x) - \varphi(y))}{|x - y|^{N+2s}} \, \mathrm{d}x \mathrm{d}y = \alpha \int_\Omega w(x) \varphi(x) \, \mathrm{d}x + \int_\Omega \varphi(x) \, \mathrm{d}x, \quad  \varphi \in X_0^s(\Omega).
\end{equation}
Actually, the existence and uniqueness of $\mathtt{w}$ is ensured via the Lax-Milgram theorem, since the bilinear form
$$
\mathcal{B}(w,\varphi)=\dfrac{\gamma(N,s)}{2} \inte_{\R^N\times \R^N} \frac{(w(x) - w(y))(\varphi(x) - \varphi(y))}{|x - y|^{N+2s}} \, \mathrm{d}x \mathrm{d}y - \alpha \int_\Omega w(x) \varphi(x) \, \mathrm{d}x 
$$
is continuous and coercive on $X_0^s(\Omega)\times X_0^s(\Omega)$.
We explicitly observe that the coercivity of $\mathcal{B}$ is trivial when $\alpha <0$, while if $0<\alpha<\lambda_1(\Omega)$, it is enough to observe that, for any $u\in X_{0}^{s}(\Omega)$, we have
\[
[w]_{H^s(\R^N)}^{2}-\alpha\int_{\Omega}|w|^{2}dx\geq(1-\alpha\left(\lambda_{1}(\Omega)\right)^{-1})
[w]_{H^s(\R^N)}^{2}.
\]

\bigskip

\begin{remark}\label{remreg}
We can argue as in Remark \ref{regularity} getting that $\mathtt{w}$ is bounded and $\mathtt w\in C^{\infty}(\Omega)$.
\end{remark}
\begin{lemma}\label{positivity}
Let $-\infty<\alpha<\lambda_{1}(\Omega)$ and $\mathtt w$ be the solution to problem \eqref{torgen_pb}. Then \(\mathtt w \ge 0\) in \(\Omega\).
\end{lemma}
\begin{proof}
Taking the negative part \(\mathtt w_- := \max\{-\mathtt w, 0\}\) as a test function in \eqref{torgen_weak}, we obtain
\[
\dfrac{\gamma(N,s)}{2}\inte_{\R^N\times \R^N}\frac{(\mathtt w(x) - \mathtt w(y))(\mathtt w_-(x) - \mathtt w_-(y))}{|x - y|^{N+2s}} \, \mathrm{d}x \mathrm{d}y 
= \alpha \int_\Omega \mathtt w(x) \mathtt w_-(x) \, \mathrm{d}x + \int_\Omega \mathtt w_-(x) \, \mathrm{d}x.
\]
Since
\[
(\mathtt w(x) - \mathtt w(y))(\mathtt w_-(x) - \mathtt w_-(y)) \le -|\mathtt w_-(x) - \mathtt w_-(y)|^2,
\]
then
\[
\alpha \int_\Omega \mathtt w(x) \mathtt w_-(x) \, \mathrm{d}x + \int_\Omega \mathtt w_-(x) \, \mathrm{d}x\leq -[\mathtt w_-]_{H^s(\R^N)}^2,
\]
and, since $\alpha<\lambda_{1}(\Omega)$, recalling \eqref{eigen} we get
\begin{align*}
\int_\Omega \mathtt w_-(x) \, \mathrm{d}x
&\leq \alpha \int_\Omega |\mathtt w_-(x)|^{2} \, \mathrm{d}x-
[\mathtt w_-]_{H^s(\R^N)}^2\\
&\leq -\left([\mathtt w_-]_{H^s(\R^N)}^2-\lambda_1(\Omega)\int_{\Omega}|\mathtt w_-(x)|^{2}\di x\right)\\
&\leq 0
\end{align*}
and we conclude $\mathtt w_{-}\equiv 0$.
\end{proof}

Furthermore, from  \eqref{torgen_max}-\eqref{torgen_pb}-\eqref{torgen_weak}, it follows that
\begin{equation}
    \label{int_torgen}
Q(\alpha, \Omega) = \int_\Omega \mathtt{w}(x) \, \mathrm{d}x,
\end{equation}
and, when $\alpha=0$, then
\[
 Q(0, \Omega)=T(\Omega).
\]
From Lemma \ref{positivity} we deduce that $Q(\alpha,\Omega)\geq0$.
The following proposition summarizes fundamental finiteness and monotonicity properties of $Q(\alpha,\Omega)$.

\begin{proposition}
\label{prop_Qdom}
Let $\Omega \subset \R^N$ be a bounded, open set with Lipschitz boundary. Then:
\begin{enumerate}[\rm(a)]

\item $Q(\alpha, \Omega)$ is finite if and only if
\[
-\infty < \alpha < \lambda_1(\Omega);
\]

\item if $\alpha < \lambda_1(\Omega^\sharp)$, then
\[
Q(\alpha, \Omega) \le Q(\alpha, \Omega^\sharp);
\]

\item $Q(\alpha, \Omega)$ is monotone increasing with respect to the domain, i.e.
\[
\Omega' \subset \Omega \quad\Longrightarrow\quad Q(\alpha, \Omega') \le Q(\alpha, \Omega).
\]
\end{enumerate}
\end{proposition}

\begin{proof} 
$ $

\begin{enumerate}[\rm(a)]

\item Suppose that \(Q(\alpha, \Omega) < +\infty\). If, by contradiction, \(\alpha \ge \lambda_1(\Omega)\), we could consider \(\psi = k \mathtt{u_1}\) as a test function in \eqref{torgen_max}, where \(k > 0\) is an arbitrary constant and \(\mathtt{u_1}\) is a positive eigenfunction corresponding to $\lambda_1(\Omega)$, immediately obtaining a contradiction. 

\noindent Conversely, if \(-\infty < \alpha < \lambda_1(\Omega)\), for every $\psi \in X_0^s(\Omega)$, we can estimate
\[
-[\psi]^2_{H^s(\mathbb{R}^{N})} + \alpha \int_\Omega |\psi(x)|^2 \, \mathrm{d}x + 2 \int_\Omega \psi(x) \, \mathrm{d}x
\le (\alpha - \lambda_1(\Omega)) \int_\Omega |\psi(x)|^2 \, \mathrm{d}x + 2 \int_\Omega \psi(x) \, \mathrm{d}x.
\]
Since \(\alpha - \lambda_1(\Omega) < 0\), applying Young's inequality shows that \(Q(\alpha, \Omega)\) is indeed finite.

\item The claim follows immediately from the P\'olya-Szeg\H o principle \eqref{eq_polya}.

\item The result is an immediate consequence of the definition of $Q(\alpha,\Omega)$.

\end{enumerate}
\end{proof}

We now list some fundamental regularity, monotonicity and asymptotic properties of the functional $Q(\alpha,\Omega)$ with respect to $\alpha$.
\begin{proposition}
\label{prop_Qalpha}
Let $\Omega \subset \R^N$ be a bounded, open set with Lipschitz boundary. Then $Q(\alpha, \Omega)$ is differentiable and monotone increasing with respect to $\alpha$. Moreover,  if $\mathtt w$ solves \eqref{torgen_pb}, then
\begin{equation*}
    \label{der_Qa}
\frac{\mathrm{d}}{\mathrm{d}\alpha} Q(\alpha, \Omega) = \int_\Omega |\mathtt w(x)|^2 \,\mathrm{d}x.
\end{equation*}
Furtheremore, it holds
\begin{equation}
\label{limQ_0}
\lim_{\alpha \to -\infty} Q(\alpha, \Omega) = 0,
    \end{equation}
\begin{equation}
\label{limQ_i}
\lim_{\alpha \to \lambda_1(\Omega)^-} Q(\alpha, \Omega) = +\infty.
\end{equation}
\end{proposition}
\begin{proof} 
The monotonicity of $Q(\alpha, \Omega)$ with respect to $\alpha$ immediately follows from the definition. We prove directly the derivation formula. For $\varepsilon>0$ small enough, let $\mathtt w_\varepsilon$ be the solution to the following problem
\begin{equation*}
\left\{
\begin{array}
[c]{lll}%
\left( -\Delta\right)^{s}w_\varepsilon=(\alpha+\varepsilon) w_\varepsilon+1 & & \text{in }%
\Omega,\\
\\
w_\varepsilon=0 & & \text{on }\R^{N}\setminus\Omega,
\end{array}
\right. %
\end{equation*}
whose weak formulation reads as
\begin{equation}
\label{torgen_weak_eps}
\dfrac{\gamma(N,s)}{2} \inte_{\R^N \times \R^N}
\frac{(w_\varepsilon(x) - w_\varepsilon(y))(\varphi(x) - \varphi(y))}{|x - y|^{N+2s}} 
\, \mathrm{d}x \mathrm{d}y
= (\alpha + \varepsilon) \int_\Omega w_\varepsilon(x) \varphi(x) \, \mathrm{d}x 
+ \int_\Omega \varphi(x) \, \mathrm{d}x, 
\quad \varphi \in X_0^s(\Omega).
\end{equation}

By taking $\varphi = \mathtt w_\varepsilon$ as a test function in the weak formulation \eqref{torgen_weak}, and $\varphi = \mathtt w$ as a test function in the weak formulation \eqref{torgen_weak_eps}, and using \eqref{int_torgen}, we obtain
\[
\begin{split}
&Q(\alpha+\varepsilon, \Omega)=\dfrac{\gamma(N,s)}{2} \inte_{\R^N \times \R^N} \frac{(\mathtt w(x) - \mathtt w(y))(\mathtt w_\varepsilon(x) - \mathtt w_\varepsilon(y))}{|x - y|^{N+2s}} \, \mathrm{d}x \mathrm{d}y -  \alpha\int_\Omega \mathtt w(x)\mathtt w_\varepsilon(x)\, \mathrm{d}x, \\
&Q(\alpha, \Omega)=\dfrac{\gamma(N,s)}{2} \inte_{\R^N \times \R^N} \frac{(\mathtt w_\varepsilon(x) - \mathtt w_\varepsilon(y))(\mathtt w(x) - \mathtt w(y))}{|x - y|^{N+2s}} \, \mathrm{d}x \mathrm{d}y - (\alpha+\varepsilon) \int_\Omega \mathtt w_\varepsilon(x)\mathtt w(x)\, \mathrm{d}x.
\end{split}
\]
Hence
\begin{equation}
\label{diffQ}
Q(\alpha+\varepsilon, \Omega)-Q(\alpha, \Omega)=\varepsilon\int_\Omega \mathtt w(x)\mathtt w_\varepsilon(x)\, \di x.
\end{equation}
Let $0<\varepsilon<\frac{\lambda_1(\Omega)-\alpha}{2}$, by Remark \ref{remreg}, there exists a constant $M>0$, independent of $\varepsilon$, such that
\[
0 \le \mathtt w_\varepsilon(x) \le M \quad \text{for all } x \in \Omega.
\]
On the other hand, the function $\mathtt w_\varepsilon-\mathtt w$ solves the problem
\begin{equation} \label{torgendiff_pb}
\left\{
\begin{array}
[c]{lll}%
\left( -\Delta\right)^{s}(w_\varepsilon-w)=\alpha (w_\varepsilon-w)+\varepsilon w_\varepsilon & & \text{in }%
\Omega,\\
\\
w_\varepsilon-w=0 & & \text{on }\R^{N}\setminus\Omega.
\end{array}
\right. %
\end{equation}
Using $\mathtt w_\varepsilon-\mathtt w$ as a test function in the weak formulation of \eqref{torgendiff_pb} and the variational characterization of $\lambda_1(\Omega)$ in \eqref{eigen}, we have
\begin{equation*}
[\mathtt w_\varepsilon-\mathtt w]^2_{H^{s}(\RR)}=\alpha\int_\Omega (\mathtt w_\varepsilon(x)-\mathtt w(x))^2\di x+\varepsilon\int_\Omega\mathtt w_\varepsilon(x) (\mathtt w_\varepsilon(x)-\mathtt w(x))\di x\end{equation*}
and hence
\begin{equation*}
(\lambda_1(\Omega)-\alpha)\int_\Omega (\mathtt w_\varepsilon(x)-\mathtt w(x))^2\di x\leq\varepsilon M\int_\Omega |\mathtt w_\varepsilon(x)-\mathtt w(x)|\, \di x\leq\varepsilon M|\Omega|^\frac 12\left(\int_\Omega (\mathtt w_\varepsilon(x)-\mathtt w(x))^2\, \di x\right)^\frac 12.
\end{equation*}
It follows that there exists the positive constant $C=(\lambda_1(\Omega)-\alpha)^{-2} M^2 |\Omega|$, which does not depend on $\varepsilon$, such that 
\begin{equation}\label{223}
\int_\Omega (\mathtt w_\varepsilon(x)-\mathtt w(x))^2\, \di x\le C\varepsilon^2.
\end{equation}
In particular, by H\"older inequality \eqref{223} implies
\[
\left|\int_{\Omega}\mathtt w(\mathtt w_\varepsilon-\mathtt w)\,dx\right|\leq \|\mathtt w\|_{L^{2}(\Omega)}\|\mathtt w_\varepsilon-\mathtt w\|_{L^{2}(\Omega)}\rightarrow 0
\]
thus
\begin{equation}
\label{lim_wwe}
\lim_{\varepsilon\rightarrow0}\int_\Omega \mathtt w(x)\mathtt w_\varepsilon(x)\, \di x=\int_\Omega |\mathtt w(x)|^2\, \di x.
\end{equation}
Finally, taking into account \eqref{diffQ} and \eqref{lim_wwe}, we have
\begin{equation*}
\lim_{\varepsilon\rightarrow0}\frac{Q(\alpha+\varepsilon, \Omega)-Q(\alpha, \Omega)}\varepsilon=\int_\Omega |\mathtt w(x)|^2\, \di x.
\end{equation*}

In order to prove \eqref{limQ_0}, we first show a bound for the solution $\mathtt w$ to problem \eqref{torgen_pb} when $\alpha<0$. 
Observe that $\mathtt w$ is classical in view of Remark \ref{remreg}.
Let $\bar{x}$ be a maximum point of $\mathtt w$. Then $(-\Delta)^{s}\mathtt w(\bar{x})\geq 0$ and from the equation satisfied by $\mathtt w$ we deduce
\[
\alpha\, \mathtt w(\bar{x}) + 1 \ge 0,
\]
whence  
\[
0 \le\mathtt w \le -\frac{1}{\alpha}\qquad \text{in}\ \Omega.
\]
It follows that $\mathtt w \to 0$ uniformly in $\Omega$ as $\alpha \to -\infty$, and therefore  
\[
\lim_{\alpha \to -\infty} Q(\alpha,\Omega) = \lim_{\alpha \to -\infty} \int_\Omega \mathtt w(x)\, \mathrm{d}x = 0.
\]

Finally, in order to prove \eqref{limQ_i}, we observe that, from Proposition \ref{prop_Qalpha}, the limit
\[\lim_{\alpha\rightarrow\lambda_1(\Omega)^-}Q(\alpha, \Omega)
\]
exists in view of the monotonicity with respect to $\alpha$. Using $w=k\mathtt u_1$ as a test function in \eqref{torgen_max}, where $k$ is an arbitrary positive constant and $\mathtt u_1$ is a positive eigenfunction corresponding to $\lambda_1(\Omega)$, we obtain
\begin{align*}
\displaystyle Q(\alpha, \Omega)&\ge-[k\mathtt u_1]^2_{H^{s}(\RR)}+\alpha\int_\Omega |k\mathtt u_1(x)|^2\, \di x+2\int_\Omega k\mathtt u_1(x) \di x\\
\\
\displaystyle\quad&= \bigl(\alpha-\lambda_1(\Omega)\bigr)k^2\int_\Omega \mathtt u_1(x)^2\, \di x+2k\int_\Omega \mathtt u_1(x)\, \di x\,.
\end{align*}
Letting $\alpha\rightarrow\lambda_1(\Omega)^-$, we have
\[
\lim_{\alpha\rightarrow\lambda_1(\Omega)^-}Q(\alpha, \Omega)\ge 2k\int_\Omega\mathtt u_1(x)\, \di x
\]
and from the arbitrariness of $k$ the claim follows.
\end{proof}

\begin{remark}
We note that in \cite{KJ81}, in the local case $s=1$, \eqref{limQ_i} is actually established in the stronger form
\[
\lim_{\alpha\rightarrow-\infty} -\alpha Q(\alpha,\Omega) = |\Omega|,
\]
by exploiting the explicit solution to problems of the form \eqref{torgen_pb} when \(\Omega\) is a ball. A glimpse of this behavior can also be observed in the proof of \Cref{prop_QR} (a) below, which contains related partial results.
\end{remark}

When $\Omega$ is a ball, all the results stated in \Cref{prop_Qalpha} hold true, but some further properties about the behaviour of $Q(\alpha, \Omega)$ with respect to the radius of the ball can be added. For this purpose, we introduce the function  \begin{equation}
\label{torgenrad_min}
Q^\sharp(\alpha,R)=Q(\alpha,B_{R})
\end{equation}
defined on the following set 
\[
D=\left\{ (\alpha,R): \alpha \le 0, \, R > 0\right\}
\cup\left\{(\alpha,R): \alpha>0,\,0<R<g(\alpha)\right\}
\]
being 
\[
g(\alpha)=\left(\frac{\lambda_{1}(B_{1})}{\alpha}\right)^{\frac{1}{2s}}.
\] 
Indeed, if $\alpha>0$ and $0<R<g(\alpha)$, we have 
\[
\alpha< R^{-2s}\lambda_{1}(B_{1})=\lambda_{1}(B_{R})
\]
and the value $Q^\sharp(\alpha,R)$ is finite.

Let us observe that a simple scaling argument shows that, if $\bar{\mathtt w}$ solves
\begin{equation}
\label{torgenrad_pb} 
\left\{
\begin{array}
[c]{lll}%
\left( -\Delta\right)^{s} \bar w=\alpha \bar w+1 & & \text{in }%
B_R,\\
\\
 \bar w=0 & & \text{on }\R^{N}\setminus B_R,
\end{array}
\right. %
\end{equation}
then the function \begin{equation*}
\bar{\mathtt h}(x)=\frac 1{R^{2s}}\bar{\mathtt w}(xR)
\end{equation*}
solves the problem
\begin{equation}
\label{torgenrad_pb_scal}
\left\{
\begin{array}
[c]{lll}%
\left( -\Delta\right)^{s} \bar h=\alpha R^{2s} \bar h +1 & & \text{in }%
B_1,\\
\\
 \bar h =0 & & \text{on }\R^{N}\setminus B_1.
\end{array}
\right. %
\end{equation}
As a consequence, we get
\begin{equation}\label{Qris}
Q^\sharp(\alpha,R)=R^{N+2s}Q^\sharp(\alpha R^{2s},1), \qquad (\alpha,R) \in D.
\end{equation}
The weak formulation of \eqref{torgenrad_pb_scal} reads as
\begin{equation}
\label{torgenrad_weak_scal}
\dfrac{\gamma(N,s)}{2} \inte_{\R^N\times\R^N} \frac{(\bar h (x) - \bar h(y))(\varphi(x) - \varphi(y))}{|x - y|^{N+2s}} \, \mathrm{d}x \, \mathrm{d}y
= \alpha R^{2s} \int_{B_1} \bar h(x) \varphi(x) \, \mathrm{d}x + \int_{B_1} \varphi(x) \, \mathrm{d}x, \quad  \varphi \in X_0^s(B_1).
\end{equation}

Let now describe the range of parameters that guarantee the finiteness of $Q^\sharp(\alpha,R)$, and study its behavior at the endpoints of this range.

\begin{proposition}
\label{prop_QR}
Let $Q^\sharp(\alpha,R)$ be the function defined in \eqref{torgenrad_min}. Then the following statements hold.

\noindent $\bullet$ If $\alpha\le 0$, the function $Q^\sharp(\alpha,R)$ is finite for every $R>0$. Moreover
\begin{equation}
\label{R_neg_lim}
\lim_{R\rightarrow+\infty}Q^\sharp(\alpha,R)=+\infty.
\end{equation}

\noindent $\bullet $ If $\alpha>0$, the function $Q^\sharp(\alpha,R)$ is finite if and only if
\begin{equation}
\label{R_pos_R}
0<R<{\widetilde{\!R}}\equiv\left(\frac{\lambda_1(B_1)}\alpha\right)^{\frac 1{2s}}.
\end{equation}
\noindent Moreover:
\begin{equation}
\label{R_pos_lim}
\lim_{R\rightarrow\,\widetilde{\!R}^-}Q^\sharp(\alpha,R)=+\infty.
\end{equation}

\end{proposition}
\begin{proof}
The case $\alpha=0$ is immediate.
In the case $\alpha<0$ \Cref{prop_Qdom} (a) implies that $Q^\sharp(\alpha,R)$ is finite for every $R>0$. We show that 
\eqref{limQ_0} can be slightly improved in the following form
\begin{equation}\label{lim>}
\liminf_{\alpha\rightarrow-\infty} \left(-\alpha Q^\sharp(\alpha,1)\right)>0.
\end{equation}
Let us consider the solution $\bar{\mathtt k}$ in \eqref{tor_sol} to the radial problem 
\begin{equation*} 
\left\{
\begin{array}
[c]{lll}%
\left( -\Delta\right)^{s} \bar k=1 & & \text{in }%
B_1\\
\\
\bar k=0 & & \text{on }\R^{N}\setminus B_1.
\end{array}
\right. %
\end{equation*}
Choosing $\psi=-\bar{\mathtt k}/\alpha$ as a test function in the definition \eqref{torgen_max} with $\Omega=B_1$, recalling \eqref{torgenrad_min} and using \eqref{upp_v}, we have
\begin{align*}
-\alpha Q^\sharp(\alpha,1)\ge&\frac 1\alpha[\bar{\mathtt k}]_{H^{s}(\RR)}-\int_{B_1}
\bar{\mathtt k}^2\di x +2\int_{B_1}
\bar{\mathtt k}\di x=\\
\\ 
=&\left(2+\frac 1\alpha\right)\int_{B_1}
\bar{\mathtt k}\di x-\int_{B_1}
\bar{\mathtt k}^2\di x\ge \left(2+\frac 1\alpha-C\right) 
\int_{B_1}
\bar{\mathtt k}\di x
\end{align*}
for some constant $C$ such that, in any dimension $N$, we have $2-C>0$.
Hence \eqref{lim>} follows.\\  
From \eqref{Qris} we have
\begin{equation*}
\lim_{R\rightarrow+\infty} \left(-\alpha Q^\sharp(\alpha,R)\right)=\lim_{R\rightarrow+\infty} R^N (-\alpha R^{2s}) Q^\sharp(\alpha R^{2s},1)
\end{equation*}
and \eqref{lim>} implies \eqref{R_neg_lim}. 

In the case $\alpha>0$, using again \eqref{Qris}, \Cref{prop_Qdom} (a) implies condition \eqref{R_pos_R} since $Q^\sharp(\alpha R^{2s},1)$ is finite if and only if 
\begin{equation*}
0<\alpha R^{2s}<\lambda_1(B_1).
\end{equation*}
On the other hand, \eqref{limQ_i} provides \eqref{R_pos_lim}.
\end{proof}

We now show that the functional $Q(\alpha,\Omega)$ can always be represented in terms of a ball contained in $\Omega^\sharp$.  

\begin{proposition}\label{prad}
Let $\Omega \subset \R^N$ be a bounded, open set with Lipschitz boundary.  
For every fixed $-\infty < \alpha < \lambda_1(\Omega)$, there exists a unique radius $R(\alpha) > 0$, with $B_{R(\alpha)} \subseteq\Omega^\sharp$, such that
\begin{equation*}\label{tor=rad}
    Q^\sharp(\alpha,R(\alpha))=Q\bigl(\alpha, B_{R(\alpha)}\bigr)=Q(\alpha,\Omega).
\end{equation*}
\end{proposition}

\begin{proof}
The continuity of $Q^\sharp(\alpha, R)$ and its differentiability with respect to $R$ can be easily proven by combining Proposition \ref{prop_Qalpha} with \eqref{Qris}.

\noindent Moreover, using \eqref{Qris} and \Cref{prop_Qdom} (c), we have
\begin{equation*}
\begin{split}
\dfrac \partial{\partial R}Q^\sharp(\alpha,R)=&R^{N-1+2s}\left[(N+2s)Q^\sharp(\alpha R^{2s},1)+2s\alpha R^{2s} \frac{\mathrm{d}}{\mathrm{d}\alpha} Q(\alpha R^{2s}, 1)\right]\\
=&R^{N-1+2s}\left[(N+2s)\int_{B_1}
\bar{\mathtt h}\,\di x+2s\alpha R^{2s}\int_{B_1}
\bar{\mathtt h}^2\di x\right]>0,\\
\end{split}
\end{equation*}
where $\bar{\mathtt h}$ is the solution to problem \eqref{torgenrad_pb_scal}. Hence, $Q^\sharp(\alpha,R)$ is strictly increasing with respect to $R$ for any fixed $\alpha$. 

From \eqref{Qris}, in view of the fact that, for a fixed $\alpha$, $Q^\sharp(\alpha R^{2s},1)$ goes to $Q^\sharp(0,1)$ as $R$ goes to 0, we have
\[\lim_{R\rightarrow0} Q^\sharp(\alpha,R)=\lim_{R\rightarrow0} R^{N+2s}Q^\sharp(\alpha R^{2s},1)=0.
\]
Using \Cref{prop_QR} we get the claim. 
\end{proof}

\section{The generalized Fractional Kohler-Jobin Inequality}\label{sec_compar}
In this section, we present a fundamental comparison result that will allow us to derive both the Kohler-Jobin and the reverse H\"older inequalities, highlighting their optimality and symmetry properties.

\subsection{A comparison result}

Before establishing the main comparison result, we first state the following lemma, whose proof follows the arguments in \cite{FPV,FV}. 

\begin{lemma}\label{dis1}
Let $\Omega \subset \R^N$ be a bounded, open set with Lipschitz boundary and let $-\infty <\alpha<\lambda_1(\Omega)$. Assume that $R(\alpha)$ is the unique radius determined by \Cref{prad} such that
$$Q(\alpha,\Omega)=Q(\alpha,B_{R(\alpha)}).$$
Let $\mathtt w$ be the solution to \eqref{torgen_pb} and $\bar{\mathtt w}$ be the solution to \eqref{torgenrad_pb} with $R=R(\alpha)$.
If $R^\sharp$ stands for the radius of $\Omega^\sharp$, then the following relations hold true
\begin{equation}
\label{PS_symm}
\frac{\gamma(N,s)}{2}\int_{B_r}\int_{B_r^c}\frac{\mathtt w^\sharp(x)-\mathtt w^\sharp(y)}{|x-y|^{N+2s}}\,\di x\di y
\leq \alpha\int_{B_r}\mathtt w^\sharp(x)\,\di x+|B_r|, \qquad 0\le r<R^\sharp,
\end{equation}
\begin{equation}
\label{PS_rad}
\frac{\gamma(N,s)}{2}\int_{B_r}\int_{B_r^c}\frac{\bar{\mathtt w}(x)-\bar{\mathtt w}(y)}{|x-y|^{N+2s}}\,\di x\di y
=\alpha\int_{B_r}\bar{\mathtt w}(x)\,\di x+|B_r|, \qquad 0\le r<R(\alpha).
\end{equation}
\end{lemma}

\begin{proof}
We only provide a sketch of the argument. By following Step~1 of the proof of Theorem~1.1 in \cite{FPV}, 
or alternatively Steps~1-2 in the proof of Theorem~3.1 in \cite{FV}, we directly obtain \eqref{PS_symm}. Equality \eqref{PS_rad} follows from a direct integration over the ball $B_r$ of the equation in problem \eqref{torgenrad_pb} with $R=R(\alpha)$. 
\end{proof}

We next prove a comparison for $\mathtt w$ and $\bar{\mathtt w}$ in term of their mass concentrations, that will be the key tool in the subsequent analysis.

\begin{theorem}\label{thm_comp}
Under the same assumptions as in Lemma \ref{dis1}, we have
\begin{equation}
    \label{q-comp}
\int_{B_r} \mathtt w^\sharp(x)\,\mathrm{d}x 
\le
\int_{B_r} \bar{\mathtt w}(x) \,\mathrm{d}x, \qquad r \ge 0.
\end{equation}
\end{theorem}
\begin{proof}
First of all, we observe that, in view of  \Cref{prop_Qdom}, $R(\alpha)\le R^\sharp$.
\\
For $r=|x|$ we set $\mathtt w^\sharp(r)=\mathtt w^\sharp(|x|)$, $\bar{\mathtt w}(r)=\bar{\mathtt w}(|x|)$ and we denote 
$$
W(r)=\frac{1}{r^N} \int_0^{r} \mathtt w^\sharp(\rho)\rho^{N-1}\,\di \rho, \qquad {\overline W}(r)=\frac{1}{r^N} \int_0^{r} \bar{\mathtt w}(\rho)\rho^{N-1}\,\di \rho.
$$
We recall (see \cite[eq. (5.28)]{FV}) that \eqref{PS_symm} and \eqref{PS_rad} imply 

\begin{equation}
\label{torgen_eqm}
(-\Delta)^s_{\R^{N+2}} W(r)\leq\alpha\,W(r)+\frac 1 N,\qquad 0 \le r < R^\sharp,
\end{equation}
and
\begin{equation}
\label{torgenrad_eqm}
(-\Delta)^s_{\R^{N+2}}{\overline W}(r)=\alpha\,{\overline W}(r)+\frac 1 N,\qquad 0\le r < R(\alpha).
\end{equation}
\\
Being $Q(\alpha,\Omega)= Q(\alpha,B_{R(\alpha)})$, \eqref{int_torgen} gives
\begin{equation}
\label{vincolo}
||\mathtt w||_{L^1(\Omega)}=||\bar {\mathtt w}||_{L^1\left(B_{R(\alpha)}\right)} \quad \Longleftrightarrow \quad \int_0^{R^\sharp} \mathtt w^\sharp(\rho)\rho^{N-1}\di \rho=\int_0^{R(\alpha)} \bar {\mathtt w}(\rho)\rho^{N-1}\di \rho.
\end{equation}
From \eqref{vincolo}, we get that, for $R(\alpha)\le r \le R^\sharp$,
$$
W(r)\le \frac{1}{r^N} \int_0^{R^\sharp} \mathtt w^\sharp (\rho)\rho^{N-1}\,\di \rho =\frac{1}{r^N}\int_0^{R(\alpha)}\bar{\mathtt w}(\rho)\rho^{N-1}\, \di \rho=\overline{W}(r).
$$
We want to show that
\begin{equation*}
W(r)\leq {\overline W}(r),\qquad 0 \le r<R(\alpha).
\end{equation*}
Assume by contradiction that there exists $(r_0,r_1)\subseteq [0,R(\alpha))$ such that the function $W(r)-\overline{W} (r)>0$ in $(r_0,r_1)$. 
Denote $Z=W-{\overline W}$; hence $Z^+ \not\equiv 0$.
By the consideration above, we have
\[
A:=\left\{Z>0\right\}\subset [0,R(\alpha)).
\]
From \eqref{torgen_eqm} and \eqref{torgenrad_eqm} we deduce, being $\alpha<\lambda_1 (B_{R(\alpha)})$,
$$(-\Delta)_{\R^{N+2}}^sZ(r)\leq \lambda_1 (B_{R(\alpha)}) Z(r),\qquad 
0 \le r < R(\alpha).$$
Since the first eigenvalue on the ball of radius $R(\alpha)$ is strictly increasing with respect to the dimension (see Remark \ref{dydarem}), denoted by $\lambda_{1,N+2}\left(B_{R(\alpha)}^{N+2}\right)$ the first eigenvalue of the ball $B_{R(\alpha)}^{N+2}$ with radius $R(\alpha)$ in dimension $N+2$, we can write
\begin{equation}
\label{Lap_N+2}
(-\Delta)_{\R^{N+2}}^sZ(r)< \lambda_{1,N+2}(B_{R(\alpha)}^{N+2}) Z(r) \quad \text{ in } A.
\end{equation}
Denoting by $|\cdot|_{N+2}$ the modulus in $\R^{N+2}$, we put $\mathcal{A}=\left\{x\in \R^{N+2}:|x|_{N+2}\in A\right\}$, so that (by abuse of notation) the $(N+2)$ variables function $Z=Z(|x|_{N+2})$ is positive only in $\mathcal{A}$. If we test inequality \eqref{Lap_N+2} with $Z^+$  we get a contradiction, since 
\begin{align*}
[Z^+]^2_{H^s(\R^{N+2})}& \le 2\int_{\R^{N+2}} (-\Delta)^s Z(x)\, Z^+(x) 
\di x
&< \lambda_{1,N+2}\left(B_{R(\alpha)}^{N+2}\right)||Z^+||^2_{L^2(\R^{N+2})}\le [Z^+]^2_{H^s(\R^{N+2})}.
\end{align*}
Hence, $W(r) \le {\overline W}(r)$ for every $r \in(0,R(\alpha))$, 
that is \eqref{q-comp}.
\end{proof}

\subsection{The generalized fractional Kohler-Jobin inequality}
We start by proving the following
\begin{proposition}\label{dec}
Let $\Omega \subset \mathbb{R}^N$ be a bounded, open set with Lipschitz boundary. and let $-\infty<\alpha <\lambda_1(\Omega)$. Let us denote by $R(\alpha)>0$ the radius such that
\[
Q(\alpha,\Omega) = Q(\alpha, B_{R(\alpha)}).
\] 
Then the mapping $\alpha \mapsto R(\alpha)$ is decreasing.
\end{proposition}

\begin{proof}
Using the notation \eqref{torgenrad_min}, we have
\[\dfrac\di{\di\alpha}Q(\alpha,\Omega)=\dfrac\di{\di\alpha}Q^\sharp(\alpha,R(\alpha))=
\frac\partial{\partial\alpha}Q^\sharp (\alpha,R(\alpha))+R'(\alpha) \frac\partial{\partial R}Q^\sharp (\alpha,R(\alpha)).
\]
Let $\mathtt w$ be the solution to \eqref{torgen_pb} and $\bar{\mathtt w}$ be the solution to \eqref{torgenrad_pb}; then  \Cref{prop_Qalpha} provides
$$
R'(\alpha)\dfrac\partial{\partial R}Q^\sharp(\alpha,R(\alpha))=\int_\Omega |\mathtt w(x)|^2\, \di x-\int_{B_{R(\alpha)}} |\bar{\mathtt w} (x)|^2\, \di x.
$$
By Theorem \ref{thm_comp} we have
$$
R'(\alpha)\dfrac\partial{\partial R}Q^\sharp(\alpha,R(\alpha))\le 0
$$
and, taking into account \Cref{prop_Qdom}(c), we get the claim.
\end{proof}

In the end, we prove a nonlocal version of the classical Kohler-Jobin inequality.

\begin{theorem}\label{thm_KJ}
Under the same assumptions as in Proposition \ref{dec}, we have
\[
\lambda_1(\Omega) \;\geq\; \lambda_1\bigl(B_{R(\alpha)}\bigr).
\]
\end{theorem}

\begin{proof}
We observe that, being $Q(\alpha,B_{R(\alpha)})$ finite, from Proposition \ref{prop_QR} we deduce that, for any $\alpha$,
\[R(\alpha)<\left(\frac{\lambda_1(B_1)}\alpha\right)^{\frac 1{2s}}.
\]
Hence, the monotonicity of $R(\alpha)$ implies
\[\exists \, \ell=\lim_{\alpha\rightarrow\lambda_1(\Omega)^{-}}R(\alpha)\le\left(\frac{\lambda_1(B_1)}{\lambda_1(\Omega)}\right)^{\frac 1{2s}}.
\]
 If, by contradiction,
\[
\ell<\left(\frac{\lambda_1(B_1)}{\lambda_1(\Omega)}\right)^{\frac 1{2s}},
\]
in view \Cref{prop_QR}, it would follow
\[\lim_{\alpha\rightarrow \lambda_1(\Omega)^{-}} Q(\alpha,\Omega)=\lim_{\alpha\rightarrow \lambda_1(\Omega)^{-}} Q^\sharp(\alpha,R(\alpha))<+\infty,
\]
in contrast with \eqref{limQ_i}. Then
\[\lim_{\alpha\rightarrow\lambda_1(\Omega)^{-}}R(\alpha)=\left(\frac{\lambda_1(B_1)}{\lambda_1(\Omega)}\right)^{\frac 1{2s}}
=R(\lambda_1(\Omega)),
\]
where, with an abuse of notation, $R(\lambda_1(\Omega))$ denotes the radius of the ball having the same first eigenvalue as $\Omega$. Then, the monotonicity of $R(\alpha)$ gives $R(\alpha)\ge R(\lambda_1(\Omega))$. Finally, being the first eigenvalue decreasing with respect to the inclusion of sets, we get
$$\lambda_1(B_{R(\lambda_1(\Omega))})=\lambda_1(\Omega)\ge \lambda_1(B_{R(\alpha)}).$$
\end{proof}

\begin{remark}
As in the local case, inequality \eqref{KJ_ineq} implies the Faber-Krahn inequality \eqref{FK}. Indeed, let $B_{R}$ the ball such that $T(B_R)=T(\Omega)$: from \eqref{SV} we deduce $T(B_R)=T(\Omega)\le T(\Omega^\sharp)$, hence $B_R \subseteq \Omega^\sharp$ and finally from \eqref{KJ_ineq} we get $$\lambda_1(\Omega) \ge \lambda_1(B_R) \ge \lambda_1(\Omega^\sharp).$$
\end{remark}

\section{The fractional reverse H\"older inequality}\label{pay_sec}

By adapting the same arguments used in the proof of Theorem \ref{thm_comp}, we can show a reverse H\"older inequality for eigenfunctions corresponding to the first eigenvalue $\lambda_1(\Omega)$ of a bounded, open set $\Omega \subset \R^N$ with Lipschitz boundary.
We start by fixing some notation.

Let $\mathtt u_1 > 0$ be a fixed eigenfunction corresponding to $\lambda_1(\Omega)$, that is let $\mathtt u_1$ be a solution to the following eigenvalue problem
\begin{equation}\label{PP}
\begin{cases}
(-\Delta)^su_1=\lambda_1(\Omega) \,u_1\qquad&\text{ in }\Omega,
\\
u_1=0& \text{ on }\mathbb{R}^N\setminus\Omega.
\end{cases}
\end{equation}

Let $B_{R_1}\subset \R^N$ be the ball (centered at the origin) having the same first eigenvalue as $\Omega$, that is $\lambda_1(B_{R_1})=\lambda_1(\Omega)$. 

As in the previous sections, let $\Omega^\sharp$ be the ball (centered at the origin) with the same measure as $\Omega$ and let us denote by $R^\sharp$ its radius. By the Faber-Krahn inequality (\Cref{FK_prop}) and the monotonicity of $\lambda_1$ with respect to the inclusion of sets, we immediately deduce that $R_1 \le R^\sharp$.

Let $\bar{\mathtt u}_1$ 
be the positive eigenfunction corresponding to $\lambda_1(B_{R_1})$ such that
\begin{equation}\label{normeq}
||\bar{\mathtt u}_1||_{L^1(B_{R_1})}=||\mathtt u_1||_{L^1(\Omega)}.
\end{equation}
In other words, let $\bar{\mathtt u}_1$ satisfy \eqref{normeq} and be a positive solution to the following eigenvalue problem
\begin{equation}\label{PPv}
\begin{cases}
(-\Delta)^s \bar u = \lambda_1(\Omega) \bar u & \text{in } B_{R_1}, \\
\bar u = 0 & \mbox{on }\mathbb{R}^N\setminus B_{R_1}.
\end{cases}
\end{equation}
We first prove that $\mathtt u_1 \prec \bar{\mathtt u}_1$.

\begin{proposition}\label{PR}
Let $\mathtt u_1$ and $\bar{\mathtt u}_1$ be defined as above. Then, 
\begin{equation*}\label{qq}
\int_{B_r}\mathtt u_1^\sharp(x)\, \di x \le \int_{B_r} \bar{\mathtt u}_1(x)\, \di x,\quad\quad  r\ge 0.
\end{equation*}
\end{proposition}

Before proving Proposition \ref{PR}, we state a lemma whose proof follows the arguments in \cite{FV,FPV}.

\begin{lemma}\label{dis}
Let $\mathtt u_1$ and $\bar{\mathtt u}_1$ be solutions to problems \eqref{PP} and \eqref{PPv}, respectively. Then the following inequalities hold true
\begin{equation}\label{b5}
\frac{\gamma(N,s)}{2}\int_{B_r}\int_{B_r^c}\frac{\mathtt u_1^\sharp(x)-\mathtt u_1^\sharp(y)}{|x-y|^{N+2s}}\,\di x\di y\leq 
\lambda_1(\Omega)\int_{B_r}\mathtt u_1^\sharp(x)\,\di x, \qquad 0 \le r < R^\sharp,
\end{equation}

\begin{equation}\label{b5v}
\frac{\gamma(N,s)}{2}\int_{B_r}\int_{B_r^c}\frac{\bar{\mathtt u}_1(x)-\bar{\mathtt u}_1(y)}{|x-y|^{N+2s}}\,\di x\di y=
\lambda_1(\Omega)\int_{B_r} \bar{\mathtt u}_1(x)\,\di x, \qquad 0 \le r < R_1.
\end{equation}
\end{lemma}
 
\begin{proof}[Proof of \Cref{PR}]
For $r=|x|$, we set $\mathtt u_1^\sharp(r)=\mathtt u_1^\sharp(x)$, $\bar{\mathtt u}_1(r)=\bar{\mathtt u}_1(|x|)$ and we denote 
$$
U(r)=\frac{1}{r^N} \int_0^{r} \mathtt u_1^\sharp(\rho)\rho^{N-1}\,\di \rho, \qquad \bar U(r)=\frac{1}{r^N} \int_0^{r} \bar{\mathtt u}_1(\rho)\rho^{N-1}\,\di \rho.
$$
As observed in \cite{FV}, \eqref{b5} and \eqref{b5v} imply
\begin{equation*}\label{PU}
(-\Delta)^s_{\R^{N+2}} U(r)\leq\lambda_1(\Omega)\,U(r)\qquad 0\le r < R^\sharp,
\end{equation*}
and
\begin{equation*}\label{PV}
(-\Delta)^s_{\R^{N+2}}\bar U(r)=\lambda_1(\Omega)\,\bar U(r)\qquad 0\le r < R_1.
\end{equation*}
We observe that, in view of \eqref{normeq}, we have
\begin{equation*}\label{11}
\int_0^{R^\sharp} \mathtt u_1^\sharp(\rho)\rho^{N-1}\di \rho=\int_0^{R_1} \bar{\mathtt u}_1
(\rho)\rho^{N-1}\di \rho.
\end{equation*}
Then, for $R_1\le r\le R^\sharp$, it holds
\begin{equation*}\label{2}
U(r)\le \frac1{r^N}\int_0^{R^\sharp} \mathtt u_1^\sharp(\rho)\rho^{N-1}\di \rho=\frac1{r^N}\int_0^{R_1} \bar{\mathtt u}_1(\rho)\rho^{N-1}\di \rho=\bar U(r).
\end{equation*}
We want to show that
\begin{equation}\label{2.}
U(r)\leq \bar U(r),\qquad\quad 0\leq r\le R_1.
\end{equation}
Assume that there exists $(r_0,r_1)\subseteq [0,R)$ such that the function $U(r)-\bar U(r)>0$ in $(r_0,r_1)$. Arguing, step by step, as in the proof of \Cref{thm_comp}, we get a contradiction and  \eqref{2.} follows.
\end{proof}

We are now ready to state the fractional reverse H\"older inequality.

\begin{theorem}
Let $\Omega \subset \mathbb{R}^N$ be a bounded, open set with Lipschitz boundary, and let $\mathtt u_1$ be an eigenfunction corresponding to $\lambda_1(\Omega)$.
Then, for any $1<q\le+\infty$,  we get
\begin{equation}\label{revhol}
\|\mathtt u_1\|_{L^q(\Omega)} \leq C\lambda_1(\Omega)^{\frac N{2s}\left(1-\frac1q\right)} \|\mathtt u_1\|_{L^1(\Omega)},
\end{equation}
where, denoted by $\bar{ \mathtt z}_1$ any first eigenfunction of the fractional Dirichlet-Laplacian in the unitary ball $B_1$, 
\begin{equation}\label{const}
C=C(N,s,q)=\lambda_1(B_1)^{\frac{N}{2s}\left(\frac 1 q -1 \right)}
\frac{||\bar{\mathtt z}_1||_{L^q(B_1)}}{||\bar{\mathtt z}_1||_{L^1(B_1)}}.
\end{equation}
\end{theorem}
\begin{proof}
With the notation used in Proposition \ref{PR}, using Proposition \ref{Propconves}, we have
\begin{equation}\label{ss}
||\mathtt u_1||_{L^q(\Omega)}\le ||\bar{\mathtt u}_1||_{L^q(B_{R_1})}=\frac{||\bar{\mathtt u}_1||_{L^q(B_{R_1})}}{||\bar{\mathtt u}_1||_{L^1(B_{R_1})}} \, ||\mathtt u_1||_{L^1(\Omega)}.
\end{equation}
We choose 
$$\bar{\mathtt u}_1(x)=\bar{\mathtt z}_1\left(\frac  x R_1 \right)$$
and we get
\begin{equation}\label{sss}
\frac{||\bar{\mathtt u}_1||_{L^q(B_{R_1})}}{||\bar{\mathtt u}_1||_{L^1(B_{R_1})}} =R_1^{N\left(\frac 1 q -1\right)}\frac{||\bar{\mathtt z}_1||_{L^q(B_1)}}{||\bar{\mathtt z}_1||_{L^1(B_1)}}.
\end{equation}
Recalling \eqref{eigen-scal}, we have
$$
\lambda_1(\Omega)=\lambda_1(B_{R_1})=\frac{\lambda_1(B_1)}{R_1^{2s}},
$$
and the claim immediately follows.
\end{proof}
\begin{remark}
As observed in \cite{C1} in the local case, we note that the Faber-Krahn type inequality \eqref{FK} is contained in \eqref{revhol}. Indeed, from \eqref{revhol}--\eqref{const}, using H\"older inequality, we immediately deduce
\begin{equation}\label{dd}
|\Omega|^{\frac 1 q -1} \le (N \omega_N)^{\frac 1 q -1} \left(\frac{\lambda_1(\Omega)}{\lambda_1(B_1)}\right)^{\frac{N}{2s}\left(1-\frac 1 q \right)}\frac{\left(\dint_0^1 \bar{\mathtt z}_1(\rho)^q \rho^{N-1}\, \di \rho\right)^{\frac 1 q }}{\dint_0^1 \bar{\mathtt z}_1 (\rho) \rho^{N-1}\, \di \rho},
\end{equation}
that is
\begin{equation}\label{ddd}
\lambda_1(\Omega) \ge \left(\frac{\omega_n}{|\Omega|}\right)^{\frac{2s}{N}}\lambda_1(B_1) \,\frac{\left(N\dint_0^1 \bar{\mathtt z}_1(\rho)\rho^{N-1}\, \di \rho\right)^{\frac{q}{q-1}\frac{2s}{N}}}{\left(N\dint_0^1\bar{\mathtt z}_1(\rho)^q \rho^{N-1}\, \di \rho\right)^{\frac{1}{q-1}\frac{2s}{N}}}.
\end{equation}
Setting $f(r)=\left(N\int_0^1 \mathtt z_1(\rho)^r \rho^{N-1}\, \di \rho\right)^{\frac 1 r}$ for $r \ge 1$, inequality \eqref{ddd} becomes
\begin{equation}\label{ee}
\lambda_1(\Omega) \ge \left(\frac{\omega_N}{|\Omega|}\right)^{\frac{2s}{N}}\lambda_1(B_1) \,\left(\frac{f(1)}{f(q)}\right)^{\frac{q}{q-1}\frac{2s}{N}}.
\end{equation}
It is easy to check that 
\begin{equation}\label{eee}
\sup_{q \ge 1}\left(\frac{f(1)}{f(q)}\right)^{\frac{q}{q-1}\frac{2s}{N}}=1.
\end{equation}
Therefore, inequalities \eqref{ee}-\eqref{eee} together give
$$
\lambda_1(\Omega) \ge \left(\frac{\omega_N}{|\Omega|}\right)^{\frac{2s}{N}}\lambda_1(B_1)=\lambda_1(\Omega^\sharp).
$$
\end{remark}

\section*{Acknowledgments}

The authors were partially supported by PRIN 2017 ``Direct and inverse problems for partial differential equations: theoretical aspects and applications'' and by Gruppo Nazionale per l'Analisi Matematica, la Probabilit\`a e le loro Applicazioni (GNAMPA) of Istituto Nazionale di Alta Matematica (INdAM). {All the authors are members of GNAMPA of INdAM.}

{V.F. and G.P. were partially supported by ``Partial differential equations and related geometric-functional inequalities'' project, CUP E53D23005540006, - funded by European Union - Next Generation EU within the PRIN 2022 program (D.D. 104 - 02/02/2022 Ministero dell'Universit\'a e della Ricerca).}

{V.F. was partially supported by ``Linear and Nonlinear PDE's: New directions and Applications'' project, CUP  E53D23018060001, - funded by European Union - Next Generation EU within the PRIN 2022 PNRR program (D.D. 1409 - 14/09/2022 Ministero dell'Universit\'a e della Ricerca).}

{B.V. was partially supported by the ``Geometric-Analytic Methods for PDEs and Applications (GAMPA)'' project, CUP I53D23002420006, - funded by European Union - Next Generation EU within the PRIN 2022 program (D.D. 104 - 02/02/2022 Ministero dell'Universit\'a e della Ricerca).}

This manuscript reflects only the authors' views and opinions and the Ministry cannot be considered responsible for them.

\end{document}